\documentclass[11pt]{article}
\usepackage{amsmath,amsthm, amssymb, amsfonts,latexsym,amscd}
\usepackage[dvips]{graphicx}
\usepackage{tikz}
\usetikzlibrary{patterns}
\newtheorem{theorem}{Theorem}[section]
\newtheorem{corollary}[theorem]{Corollary}
\newtheorem{lemma}[theorem]{Lemma}

\begin{document}
\title{\bf The core index of a graph}
\author{Dinesh Pandey\footnote{Supported by UGC Fellowship scheme (Sr. No. 2061641145), Government of India} \and Kamal Lochan Patra}
\date{}
\maketitle

\begin{abstract}
For a graph $G,$ we denote the number of connected subgraphs of $G$ by $F(G)$. For a tree $T$, $F(T)$ has been studied extensively and it has been observed that $F(T)$ has a reverse correlation with Wiener index of $T$. Based on that, we call $F(G),$  the core index of $G$.

In this paper, we characterize the graphs which extremize the core index among all graphs on $n$ vertices with $k\geq 0$ connected components. We extend our study of core index to unicyclic graphs and connected graphs with fixed number of pendant vertices. We obtained the unicyclic graphs which extremize the core index over all unicyclic graphs on $n$ vertices. The graphs which extremize the core index among all unicyclic graphs with fixed girth are also obtained. Among all connected graphs on $n$ vertices with fixed number of pendant vertices, the graph which minimizes and the graph which maximizes the core index are characterized.\\

\end{abstract}

\section{Introduction}
Throughout this paper, graphs are finite, simple and undirected. Let $G$ be a connected graph with vertex set $V=V(G)$ and edge set $E=E(G).$ We denote by $d(v)$ the degree of a vertex $v\in V$. A tree is a connected acyclic graph. A binary tree is a tree $T$ such that every vertex of $T$ has degree $1$ or $3.$ A vertex of degree one is called a {\it pendant} vertex of $G$(it is referred as a leaf if $G$ is a tree). For $u,v\in V,$  the {\it distance} between $u$ and $v$ in $G$, denoted by $d_G(u,v)$, is the number of edges on the shortest path connecting $u$ and $v$. For $v\in V,$ the distance of $v$, denoted by $d'_G(v)$, is defined as 
$d'_G(v)=\underset{u\in V}{\sum}d_G(v,u).$ 

The {\it Wiener index} of $G$, denoted by $W(G),$ is defined as the sum of distances of all unordered pairs of vertices of $G$. So, $W(G)=\frac{1}{2}\underset{v\in V}{\sum}d'_G(v).$ 

Let $\mathbb{N}$ be the set of natural numbers. For a given connected graph $G$, let $f_G:V\rightarrow\mathbb{N}$ be the function defined by $v\mapsto f_G(v)$, where $f_G(v)$ is the number of connected subgraphs of $G$ containing $v$.  The \textit{subgraph core} of $G$, denoted by $S_c(G),$ is defined as the set of vertices maximizing $f_G(v).$ The concept of subgraph core of a graph is first defined by Szekely and Wang in \cite{sw} for trees only. They called it as subtree core and proved some interesting results on $S_c(G),$ when $G$ is a tree. For a tree $T$, they proved that the function $f_T$ is strictly concave in the following sense(see the proof of Theorem 9.1 of \cite{sw}) 
\begin{lemma}\cite{sw}\label{concave}
If $u,v,w$ are three vertices of a tree $T$ with $\{u,v\},\{v,w\}\in E(T)$, then $2f_T(v)-f_T(u)-f_T(w)>0.$
\end{lemma}
Using Lemma \ref{concave}, Szekely and Wang proved the following result(\cite[Theorem 9.1]{sw}).

\begin{theorem}\label{thm-score}(\cite{sw},Theorem 9.1)
The subgraph core of a tree consists of either one vertex or two adjacent vertices.
\end{theorem}

We denote the number of connected subgraphs of $G$ by $F(G)$. For a disconnected graph $H$  with connected components $H_1,H_2,\ldots, H_k$, we have $F(H)=\underset{i=1}{\overset{k}{\sum}}F(H_i).$  In \cite{sw}, the authors have first proved some extremization(both minimization and maximization) results on the number of subtrees of trees. Following is the first result in this regard.

\begin{theorem}(\cite{sw}, Theorem 3.1)\label{tree-F}
Among all trees on $n$ vertices, The number of subtrees is maximized by the star $K_{1,n-1}$ and minimized by the path $P_n.$ Moreover, if $T$ is a tree on $n$ vertices then ${n+1 \choose 2}\leq F(T)\leq 2^{n-1}+n-1$ with left and right equalities happen for $P_n$ and $K_{1,n-1}$, respectively.
\end{theorem}

Since then a lot has been studied for the number of subtrees of trees(see \cite{yan,sw1,kir,zha1,sli,zha,sil,zha2,xia,che}). It is also observed that there is a reverse correlation between number of subtrees and Wiener index of trees. Based on that, for any graph $G,$ we call the number $F(G)$ as the \textit{core index} of $G.$ The problem of studying the core index of trees has received much attention. Many results have been obtained so far related to the core index of trees. Following are some extremization results associated with core index of trees. 

Szekely and Wang in \cite{sw}, obtained the extremal tree which minimizes the core index among all binary trees with $k$ leaves. The extremal tree which maximizes the core index among all binary trees with $k$ leaves is obtained in \cite{sw1}. The extremization results on the core index of trees with a given degree sequence are studied in \cite{zha} and \cite{zha2}. In \cite{zha}, the authors have also obtained the extremal trees which maximizes the core index among all trees on $n$ vertices with fixed independence number(matching number). Yan and Yeh in \cite{yan}, obtained the extremal tree which minimizes the core index among all trees on $n$ vertices with fixed maximal degree. The extremal tree which maximizes the core index among all trees on $n$ vertices with fixed maximal degree is determined by Kirk and Wang in \cite{kir}. In \cite{yan} and \cite{che}, the authors have studied the core index of trees on $n$ vertices with fixed diameter. Some more extremization results on the core index of trees by fixing different graph theoretic constraints are obtained by  Li and Wang in \cite{sli}. In this article we concentrate on the core index of arbitrary graphs.
 
In  Section 2, we study the extremization of the core index of graphs on $n$ vertices. In Section 3, we examine the effect on the core index of graphs obtained by different graph perturbations. In Section 4, the core index of unicyclic graphs are studied. In the last Section, we discuss about the core index of connected graphs on $n$ vertices with $k$ pendant vertices.

\section{Graphs on $n$ vertices}

The main goal of this section is to extremize the core index among all graphs on $n$ vertices. Let $h_k$ be the number of connected graphs on $k$ vertices. Then $h_k$ can be obtained by the recurrence relation $k2^{{k \choose 2 }}=\underset{i}{\sum}{k \choose i }ih_i2^{{k-i \choose 2}}$ (see Theorem 3.10.1 of \cite{wil}). So, $h_1=1,h_2=1,h_3=4,\cdots$ and the sequence $h_k,\; k\geq 2$ is strictly increasing.

\begin{theorem}\label{con-F}
Among all connected graphs on n vertices, the core index is maximized by the complete graph $K_n$ and minimized by the path $P_n.$ Moreover, if $G$ is a connected graph on $n$ vertices then ${n+1 \choose 2}\leq F(G)\leq M,$ where $M=\underset{i=1}{\overset{n}{\sum}}{n \choose i}h_i.$ 
\end{theorem}
\begin{proof}
Let $G$ be a connected graph on n vertices. If G is not complete then G is a proper subgraph of $K_n$. So,  $F(G)<F(K_n)$. If $G$ is not a tree then $G$ contains atleast one cycle. Delete edges from cycles of $G$ so that it becomes a tree, say $T$. So $T$ is a proper subgraph of $G$ and hence $F(T)<F(G)$. But by Theorem \ref{tree-F}, $F(P_{n})\leq F(T)$ with equality happens if $T$ is a path and also $F(P_n)={n+1 \choose 2}$. 

 For $1\leq i \leq n,$ let $S_i$ be the set of all connected subgraphs of $K_n$ on $i$ vertices. Then $|S_i|={n \choose i}h_i.$ So, $F(K_n)=\underset{i=1}{\overset{n}{\sum}}{n \choose i}h_i.$ This completes the proof.
\end{proof}

Let $G$ be a graph with $n$ vertices. Let $u$ and $v$ be two non-adjacent vertices of $G.$ If $G'=G\cup \{u,v\}$ then  $F(G)<F(G').$ This leads to the following: among all graphs on $n$ vertices, the complete graph $K_n$ maximizes the core index and the graph $\overline{K_n}$(the compliment of $K_n$) minmizes the core index. Furthermore, if $G$ is a graph on $n$ vertices, then $n\leq F(G)\leq M=\underset{i=1}{\overset{n}{\sum}}{n \choose i}h_i.$

We now consider the problem of extremizing the core index of graphs on $n$ vertices with $k$ connected components. Let $G_1$ and $G_2$ be two graphs with disjoint vertex set $V_1$ and $V_2$ and edge set $E_1$and $E_2$ respectively. The union $G_1\cup G_2$ is the graph with vertex set $V_1\cup V_2$ and edge set $E_1\cup E_2.$ We denote the union of $k$ copies of the graph $G$ by $kG.$

\begin{lemma}\label{con-l1}
For $2 \leq l \leq m$, $F(K_l \cup K_m)<F(K_{l-1} \cup K_{m+1})$ .
\end{lemma}

\begin{proof}
We denote the number of $r$ permutations  of $n$ elements by $P(n,r).$ By the definition of core index, $F(K_l \cup K_m)=F(K_l)+F(K_m)$ and by Theorem \ref{con-F}, $F(K_l)=\sum_{i=1}^{l} {l \choose i}h_i$, where $h_i$ is the number of connected graphs on $i$ vertices. Then
\begin{align*}
&F(K_{l-1} \cup K_{m+1})-F(K_l \cup K_m)\\&=F(K_{l-1})+F(K_{m+1})-F(K_l)-F(K_m)\\
&=\sum_{i=1}^{l-1} {l-1 \choose i}h_i+\sum_{i=1}^{m+1} {m+1 \choose i}h_i-\sum_{i=1}^{l} {l \choose i}h_i-\sum_{i=1}^{m} {m \choose i}h_i\\
&=\sum_{i=1}^{m} \left({m+1 \choose i}-{m \choose i}\right)h_i +\sum_{i=1}^{l-1} \left({l-1 \choose i}-{l \choose i}\right)h_i+(h_{m+1}-h_l)\\
& > \sum_{i=1}^{m} \frac{i}{m-i+1} {m\choose i}h_i - \sum_{i=1}^{l-1}\frac{i}{l-i}{l-1\choose i}h_i\\
&=\sum_{i=1}^{m} \frac{1}{(i-1)!} P(m,i-1) h_i -\sum_{i=1}^{l-1} \frac{1}{(i-1)!}P(l-1,i-1) h_i\\
&> \sum_{i=1}^{l-1}\frac{1}{(i-1)!}(P(m,i-1)-P(l-1,i-1) )h_i\\
&> 0.
\end{align*}
\end{proof}

\begin{theorem}\label{con-T1}
Let $G$ be a graph on $n$ vertices with $k$ connected components. Then $$F(G)\leq (k-1)+\sum_{i=1}^{n-k+1} {n-k+1 \choose i}h_i$$ and the equality happens if and only if   $G=(k-1)K_1\cup K_{n-k+1}.$
\end{theorem}
 
\begin{proof}
Let $G_1,G_2,\ldots, G_k$ be the $k$ components of $G$ with $|V(G_i)|=l_i$ for $i=1,\ldots,k.$ Then by Theorem \ref{con-F}, $F(G)=F(G_1\cup \cdots \cup G_k)\leq F(K_{l_1}\cup \cdots \cup K_{l_k})$ and the equality happens if and only if $G_i=K_i;\ 1\leq i\leq k.$ Now the proof follows from Lemma \ref{con-l1}.
\end{proof} 

\begin{lemma}\label{con-l2}
For $2 \leq l \leq m$, $F(P_l \cup P_m)<F(P_{l-1}\cup P_{m+1})$.
\end{lemma}
\begin{proof} 
By Thorem \ref{tree-F}, $F(P_m)=\frac{m(m+1)}{2}.$ So We have
\begin{align*}
&F(P_{l-1}\cup P_{m+1})-F(P_{l}\cup P_{m})\\
&=F(P_{l-1})+F(P_{m+1})-F(P_{l})-F(P_{m})\\
&=\frac{(l-1)l}{2}+\frac{(m+1)(m+2)}{2}-\frac{l(l+1)}{2}-\frac{m(m+1)}{2}\\
&=m+1-l>0.
\end{align*}

\end{proof}

\begin{theorem}\label{con-T2}
For positive integers $n$ and $k$, let $n=kq+r$ where  $q,r \in \mathbb{Z} $ and $0\leq r < k$ also let $G$ be a graph on $n$ vertices with $k$ connected components. Then $$F(G)\geq r(q+1)+\frac{kq(q+1)}{2}$$ and the equality happens if and only if   $G=rP_{q+1}\cup(k-r)P_q.$ 
\end{theorem}

\begin{proof}
Let $G_1,G_2,\ldots, G_k$ be the $k$ components of $G$ with $|V(G_i)|=l_i$ for $i=1,\ldots,k.$ Then by Theorem \ref{con-F}, $F(G)=F(G_1\cup \cdots \cup G_k)\geq F(P_{l_1}\cup \cdots \cup P_{l_k})$ and the equality happens if and only if $G_i=P_{l_i},\; 1\leq i\leq k.$ By Lemma \ref{con-l2}, $F(G)\geq F(rP_{q+1}\cup(k-r)P_q)$ and the equality happens if and only if   $G=rP_{q+1}\cup(k-r)P_q.$ 

Also we have $F(rP_{q+1}\cup(k-r)P_q)=r\frac{(q+2)(q+1)}{2}+(k-r)\frac{q(q+1)}{2}=r(q+1)+\frac{kq(q+1)}{2}.$ This completes the proof
\end{proof}

Theorem \ref{con-T2} also proves that among all acyclic graphs on $n$ vertices with $k$ components, the core index is minimized by the graph $rP_{q+1}\cup(k-r)P_q.$ Next we will prove the maximization case for the acyclic graphs.
\begin{lemma}\label{con-l3}
For $2 \leq l \leq m$, $F(K_{1,l-1} \cup K_{1,m+1})>F(K_{1,l }\cup K_{1,m})$ 
\end{lemma}
\begin{proof}
By Theorem \ref{tree-F}, $F(K_{1,m-1})=2^{m-1}+m-1.$ So We have
\begin{align*}
&F(K_{1,l-1} \cup K_{1,m+1})-F(K_{1,l} \cup K_{1,m})\\
&=F(K_{1,l-1})+F(K_{1,m+1})-F(K_{1,l})-F(K_{1,m})\\
&=2^{l-1}+l-1+2^{m+1}+m+1-2^l-l-2^m-m\\
&=2^m-2^{l-1}>0
\end{align*}

\end{proof}

\begin{theorem}\label{con-T3}
Let $G$ be a acyclic graph on $n$ vertices with $k$ connected components. Then $$F(G)\leq 2^{n-k}+n-1$$ and the equality happens if and only if   $G=(k-1)K_1\cup K_{1, n-k}.$
\end{theorem}
\begin{proof}
Let $G_1,G_2,\ldots, G_k$ be the $k$ components of $G$ with $|V(G_i)|=l_i$ for $i=1,\ldots,k.$ Then by Theorem \ref{tree-F}, $F(G)=F(G_1\cup \cdots \cup G_k)\leq F(K_{1,l_1-1}\cup \cdots \cup K_{1,l_k-1})$ and the equality happens if and only if $G_i=K_{1,l_i-1};\ 1\leq i\leq k.$ By Lemma \ref{con-l3}, $F(G)\leq F((k-1)K_1\cup K_{1, n-k})$ and the equality happens if and only if   $G=(k-1)K_1\cup K_{1, n-k}.$ 

Also we have $F((k-1)K_1\cup K_{1, n-k})=k-1+2^{n-k}+n-k=2^{n-k}+n-1.$ This completes the proof.

\end{proof}
 
\section{Effect on the core index by some graph perturbations}

In this section, we will prove some lemmas which are very useful in proving the main results of this paper. Let $G$ be a graph and let $v_1,v_2,\ldots, v_k \in V(G).$ We denote by $f_G(v_1,v_2,\ldots,v_k)$ the number of connected subgraphs of $G$ containing $v_1,v_2,\ldots,v_k.$

\begin{lemma}\label{effect-1}
Let $G$ be a connected graph and let $\{u,v \}$ be a bridge in $G$ such that neither  $u$ nor $v$ be a pendent vertex. Let $G'$ be the graph obtained from $G$ by identifying the vertices $u$ and $v$ and adding a pendent vertex $y$ at the identified vertex of $G'.$ Then $F(G') > F(G)$.
\end{lemma}
\begin{proof}
Let $G_1$ and $G_2$ be two subgraphs of $G$ containing $u$ and $v$ respectively, after deleting the edge  $\{ u,v \}$ from $G.$  Let $w$ be the vertex in $G'$ obtained from identifying the vertices $u$ and $v$(see Figure \ref{fig1}).\\

\vspace{0.75in}
\begin{figure}[!h]
 \includegraphics[scale=0.75]{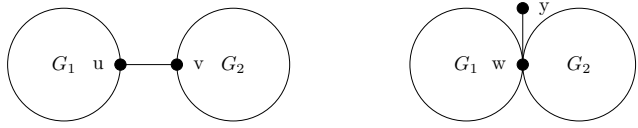}
\caption{Identifying two vertices and adding a pendant vertex there}\label{fig1}
\end{figure}

Then we have, $$F(G)=F(G_1)+F(G_2)+f_{G_1}(u)f_{G_2}(v)$$ and
\begin{align*}
F(G')&=F(G_1)+F(G_2)-1 +(f_{G_1}(w)-1)(f_{G_2}(w)-1)\\
	   &+1+(f_{G_1}(w)-1)(f_{G_2}(w)-1)+f_{G_1}(w)+f_{G_2}(w)-1\\
       &=F(G_1)+F(G_2)+2f_{G_1}(w)f_{G_2}(w)-f_{G_1}(w)-f_{G_2}(w)+1.
\end{align*}

Here $F(G_1)+F(G_2)-1 +(f_{G_1}(w)-1)(f_{G_2}(w)-1)$ counts the number of connected subgraphs of $G'$ not containing $y$ and $1+(f_{G_1}(w)-1)(f_{G_2}(w)-1)+f_{G_1}(w)+f_{G_2}(w)-1$ counts the number of connected subgraphs of $G'$ containing $y$. As $f_{G_1}(u)=f_{G_1}(w)$ and $f_{G_2}(v)=f_{G_2}(w)$, so 

$$ F(G') - F(G)=f_{G_1}(w)f_{G_2}(w)-f_{G_1}(w)-f_{G_2}(w)+1=(f_{G_1}(w)-1)(f_{G_2}(w)-1)>0.$$ This complets the proof.

\end{proof}

\begin{corollary}
Among all trees on $n$ vertices, the star $K_{1,n-1}$ has the maximum core index.
\end{corollary}
\vspace{0.25in}
\begin{figure}[!h]
 \includegraphics[scale=0.75]{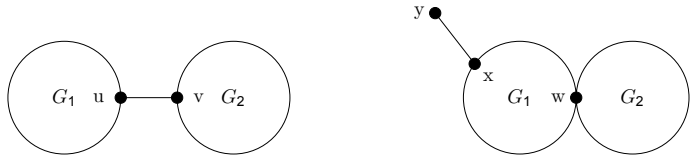}
\caption{Identifying two vertices and adding a pendant vertex}\label{fig2}
\end{figure}
\begin{lemma}\label{effect-2}
Let $G$ be a connected  graph and let $\{ u,v \}$ be a bridge in $G.$ Let $G_1$ and $G_2$ be two components of $G - \{u,v\}$  containing $u$ and $v$, respectively. Let $f_{G_1}(u) \leq f_{G_2}(v)$ and let $x\in V(G_1)$ such that $f_{G_1}(x,u)\geq 2$. Construct a graph $G'$ from $G$ by identifying the vertices $u$ and $v$ and adding a pendent vertex $y$ at $x.$ Then $F(G')>F(G)$. 
\end{lemma}
\begin{proof}
Let $w$ be the vertex in $G'$ obtained by identifying the vertices $u$ and $v$ of $G$ (see Figure \ref{fig2}). Then we have 
$$F(G)=F(G_1)+F(G_2)+f_{G_1}(u)f_{G_2}(v)$$ and
$$F(G')=F(G_1)+F(G_2)-1+(f_{G_1}(w)-1)(f_{G_2}(w)-1)+1+f_{G_1}(x)+f_{G_1}(x,w)(f_{G_2}(w)-1).$$
Here $F(G_1)+F(G_2)-1 +(f_{G_1}(w)-1)(f_{G_2}(w)-1)$ counts the number of connected subgraphs of $G'$ not containing $y$ and $1+f_{G_1}(x)+f_{G_1}(x,w)(f_{G_2}(w)-1)$ counts the number of connected subgraphs of $G'$ containing $y$. As $f_{G_1}(x,w) \geq 2$, we have 

\begin{align*}
F(G')&\geq F(G_1)+F(G_2)+(f_{G_1}(w)-1)(f_{G_2}(w)-1)+f_{G_1}(x)+2(f_{G_2}(w)-1)\\
	  &=F(G_1)+F(G_2)+f_{G_1}(w)f_{G_2}(w)-f_{G_1}(w)-f_{G_2}(w)+f_{G_1}(x)+2f_{G_2}(w)-1.
\end{align*}

Since $f_{G_1}(u)=f_{G_1}(w)$, $f_{G_2}(v)=f_{G_2}(w)$ and $f_{G_1}(u) \leq f_{G_2}(v)$, so we have
\begin{align*}
F(G')-F(G) & \geq  2f_{G_2}(w)+f_{G_1}(x)-f_{G_1}(w)-f_{G_2}(w)-1\\
& = f_{G_2}(w)-f_{G_1}(w)+f_{G_1}(x)-1\\
& >0,  \mbox{ as $f_{G_1}(x)\geq f_{G_1}(x,w)\geq 2$ }.
\end{align*}
This completes the proof.
\end{proof}

\vspace{1.25in}
\begin{figure}[!h]
 \includegraphics[scale=0.75]{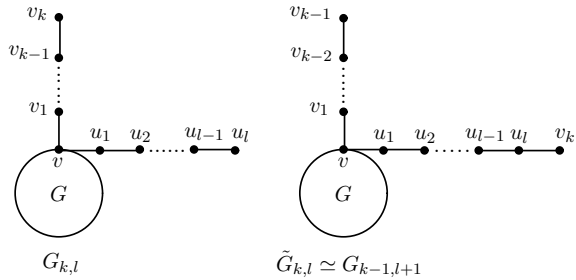}
\caption{Grafting an edge}\label{fig3}
\end{figure}

Let  $G$ be a connected graph on $n\geq 2$ vertices. Let
$v$ be a vertex of $G.$ For $l,k \geq 1,$  let $G_{k,l}$ be the
graph obtained from $G$ by attaching two new paths
$P:vv_{1}v_{2}\cdots v_{k}$ and $Q:vu_{1}u_{2}\cdots u_{l}$ of
lengths $k$ and $l$, respectively at $v$, where
$u_{1},u_{2},\ldots,u_{l}$ and $v_{1},v_{2},\ldots,v_{k}$ are
distinct new vertices. Let ${\widetilde G}_{k,l}$ be the graph
obtained by removing the edge $\{v_{k-1},v_{k}\}$ and adding the
edge $\{u_{l},v_{k}\} $ (see Figure \ref{fig3}). Observe that the
graph ${\widetilde G}_{k,l}$ is isomorphic to the graph
$G_{k-1,l+1}$. We say that ${\widetilde G}_{k,l}$ is obtained from
$G_{k,l}$ by {\em grafting} an edge.\label{graph:gkl}

We denote by $P_n$ the path on the $n$ vertices $v_1,v_2,\ldots,v_n$, where $v_1$ and $v_n$ are pendant vertices, and for $i=2,3,\ldots, n-1,$ vertex $v_i$ is adjacent to vertices $v_{i-1}$ and $v_{i+1}$. Then for $1\leq i\leq n,$ $f_{P_n}(v_i)=i(n+1-i).$ The next lemma compares the core index of $G_{k,l}$ and $G_{k-1,l+1}.$
\begin{lemma}\label{effect-3}
Let $G$ be a connected graph on $n\geq 2$ vertices and $v \in V(G)$. Let $G_{k,l}$ be the graph as defined above. If $1\leq k\leq l$ then $F(G_{k-1,l+1})<F(G_{k,l})$.
\end{lemma}

\begin{proof}

Let $P=:vv_1v_2\cdots v_k$ and $Q=:vu_1u_2\cdots u_l$ be two paths of length $k$ and $l$, $1\leq k\leq l$, respectively, attached at the vertex $v$ of $G$. Then $S=:v_k\cdots v_1vu_1\cdots u_l$ is a path in $G_{k,l}$ and $S'=:v_{k-1}\cdots v_1vu_1\cdots u_lv_k$ is a path in $G_{k-1,l+1}.$ In $G_{k,l}$, $v$ is the only common vertex between $G$ and the path $S$ and in $G_{k-1,l+1}$ also $v$ is the only common vertex between $G$ and the path $S'.$ Then we have 
$$F(G_{k,l})=F(G)+F(S)-1+(f_{G}(v)-1)(f_{S}(v)-1)$$ and 
$$F(G_{k-1,l+1})=F(G)+F(S')-1+(f_{G}(v)-1)(f_{S'}(v)-1).$$

Both $S$ and $S'$ are path on $l+k+1$ vertices. So, $F(S)=F(S'),$ and $f_S(v)=(k+1)(l+1)$ and $f_S'(v)=k(l+2)$ 
\begin{align*}
 F(G_{k,l})-F(G_{k-1,l+1})&=(f_{G}(v)-1)(f_{S}(v)-f_{S'}(v))\\
           &=(f_{G}(v)-1)((k+1)(l+1)-k(l+2))\\
           &=(f_{G}(v)-1)(l-k+1)\\
           &>0, \mbox{ as $l\geq k$ and $|V(G)|\geq 2$ }.   
\end{align*}
This completes the proof.
\end{proof}
\begin{corollary}
Among all trees on $n$ vertices, the path $P_n$ has the minimum core index. 
\end{corollary}

We denote by $K_{1,n-1}$ the star on the $n$ vertices $v_1,v_2,\ldots,v_n$, where $v_1,v_2,\ldots,v_{n-1}$ are pendant vertices. Then for $1\leq i\leq n-1,$ $f_{K_{1,n-1}}(v_i)=1+2^{n-2}$ and $f_{K_{1,n-1}}(v_n)=2^{n-1}.$

\begin{lemma}\label{effect-4}
Let $G$ be a connected graph on $n \geq 2$ vertices and let $u,v \in V(G).$ For $n_1,n_2\geq 0,$ let $G_{uv}(n_1,n_2)$ be the graph obtained from $G$ by attaching $n_1$ pendant vertices at $u$ and $n_2$ pendant vertices at $v.$ Let $f_{G}(u) \geq f_{G}(v).$ If $n_1,n_2\geq 1,$ then $$F(G_{uv}(n_1+n_2,0)) > F(G_{uv}(n_1,n_2).$$
\end{lemma}

\begin{proof}
Let $S_1$ be the star in $G_{uv}(n_1+n_2,0)$ on $n_1+n_2+1$ vertices with $u$ is the only common vertex between $S_1$ and $G$ (see Figure \ref{fig4}). Then 
$$F(G_{uv}(n_1+n_2,0))=F(G)+F(S_1)-1 +(f_{G}(u)-1)(f_{S_1}(u)-1).$$

\vspace{0.45in}
\begin{figure}[!h]
 \includegraphics[scale=0.75]{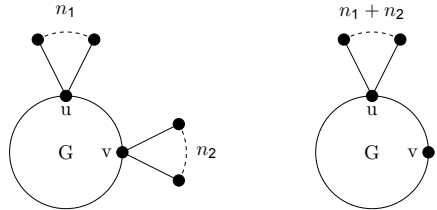}
\caption{The graphs $G_{uv}(n_1,n_2)$ and $G_{uv}(n_1+n_2,0)$ }\label{fig4}
\end{figure}

Let $S_2$ be the star in $G_{uv}(n_1,n_2)$ on $n_1+1$ vertices with $u$ is the only common vertex between $S_2$ and $G$ (see Figure \ref{fig4}) and let $S_3$ be the star in $G_{uv}(n_1,n_2)$ on $n_2+1$ vertices with $v$ is the only common vertex between $S_3$ and $G.$ Then
\begin{align*}
F(G_{uv}(n_1,n_2))&=F(G)+F(S_2)+F(S_3)-2+(f_{G}(u)-1)(f_{S_2}(u)-1)\\
    &+(f_{G}(v)-1)(f_{S_3}(v)-1)+f_{G}(u,v)(f_{S_2}(u)-1)(f_{S_3}(v)-1).
\end{align*}

Since $f_{G}(u) \geq f_{G}(v)$, so $f_{G}(u)-1 \geq f_{G}(v)-1$ and $f_{G}(u)-1 \geq f_{G}(u,v).$ Also $$F(S_1)-F(S_2)-F(S_3)+1=(2^{n_1+n_2} +n_1+n_2) - (2^{n_1}+n_1)-(2^{n_2}+n_2)+1>0$$ and
 $$f_{S_1}(u)-f_{S_2}(u)f_{S_3}(v)=2^{n_1+n_2}-2^{n_1}.2^{n_2}=0.$$
So, we have

\begin{align*}
&F(G_{uv}(n_1+n_2,0))- F(G_{uv}(n_1,n_2))\\ &=F(S_1)-F(S_2)-F(S_3)+1+(f_{G}(u)-1)(f_{S_1}(u)-1)\\
 		   &-(f_{G}(u)-1)(f_{S_2}(u)-1)-(f_{G}(v)-1)(f_{S_3}(v)-1)\\
 	       &-f_{G}(u,v)(f_{S_2}(u)-1)(f_{S_3}(v)-1)\\
 		   & > (f_{G}(u)-1)(f_{S_1}(u)-f_{S_2}(u)f_{S_3}(v))\\
 		   &= 0.
\end{align*}
This completes the proof.

\end{proof}

\begin{lemma}\label{effect-5}
Let $G$ be a connected graph on $n \geq 3$ vertices. Let $u,v \in V(G)$  such that $f_G(u,v)\geq 2.$ For $l,k\geq 1,$ let $G_{uv}^p(l,k)$ be the graph obtained from $G$ by identifying a pendant vertex of the path $P_l$ with $u$ and  identifying a pendant vertex of the path $P_k$ with $v$.  Let $f_{G}(u) \leq f_{G}(v).$ If $l,k\geq 2,$ then $$F(G_{uv}^p(l+k-1,1)) < F(G_{uv}^p(l,k)).$$
\end{lemma}

\begin{proof}
\vspace{0.90in}
\begin{figure}[!h]
 \includegraphics[scale=0.75]{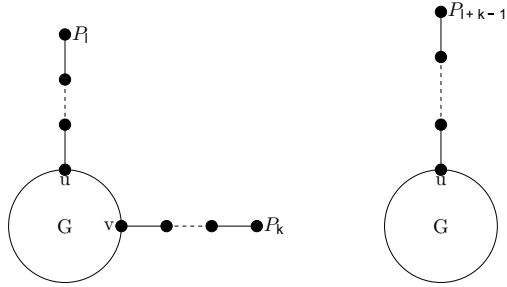}
\caption{The graphs $G_{uv}^p(l,k)$ and $G_{uv}^p(l+k-1,1)$ }\label{fig5}
\end{figure}
We have 
\begin{align*}
F(G_{uv}^p(l,k) )&=F(G)+F(P_l)+F(P_k)-2+(f_{G}(u)-1)(f_{P_l}(u)-1)\\
		&+(f_{G}(v)-1)(f_{P_k}(v)-1)+f_{G}(u,v)(f_{P_l}(u)-1)(f_{P_k}(v)-1)\\
F(G_{uv}^p(l+k-1,1) )&=F(G)+F(P_{l+k-1})-1+(f_{G}(u)-1)(f_{P_{l+k-1}}(u)-1)	
\end{align*}
 and the difference 
\begin{align*}
&F(G_{uv}^p(l,k))-F(G_{uv}^p(l+k-1,1))\\
&\geq F(P_l)+F(P_k)-F(P_{l+k-1})-1\\
		    &+(f_{G}(u)-1)(f_{P_l}(u)-1+f_{P_k}(v)-1-f_{P_{l+k-1}}(u)+1)\\
	       &+f_{G}(u,v)(f_{P_l}(u)f_{P_k}(v)-f_{P_l}(u)-f_{P_k}(v)+1)\\
	       &=l+k-lk-1+f_{G}(u,v)(lk-l-k+1)\\
	       &=(f_{G}(u,v)-1)(lk-l-k+1)\\
	       &>0, \mbox{ as $l,k\geq 2$ and $f_{G}(u,v)\geq 2$}
\end{align*}
This completes the proof.
\end{proof}

\section{Unicyclic graphs}

We denote the cycle on $n$ vertices by $C_n.$ The girth of a graph $G$ is the length of a shortest cycle in $G.$ A connected graph on $n$ vertices with $n$ edges is called unicyclic. For $n\geq 3,$ we denote by $\mathcal{U}_n$ the set of all unicyclic graphs on $n$ vertices and by $\mathcal{U}_{n,g}$ the set of all unicyclic graphs on $n$ vertices with girth $g.$ Clearly  $3\leq g\leq n$ and $\mathcal{U}_{n,g}\subseteq \mathcal{U}_n.$ If $g=n$ then $C_n$ is the only element of $\mathcal{U}_{n,n}.$ The next lemma gives the value of the core index of $C_n.$

\begin{lemma}\label{cycle-F}
Let $n\geq 3$ and let $v\in V(C_n) .$ Then $F(C_n)=n^2+1$ and $f_{C_n}(v)=2n+ {n-1\choose 2}.$  
\end{lemma}
\begin{proof}
Let $V(C_n)=\{v=v_1,v_2,\ldots,v_n\}$ and $C_n=:v_1v_2\cdots v_nv_1.$ The single vertices ${{v_1}},{v_2},\ldots,{v_n}$ are $n$ connected subgraphs of $C_n$. The cycle $C_n$ itself is one connected subgraph of $C_n$. Any other connected subgraph of $C_n$ is a path having the end vertices from ${{v_1},{v_2},\ldots,{v_n}}$. If we chose any two vertices $v_i,v_j$ from ${v_1},{v_2},\ldots,{v_n}$, it corresponds two paths, one in clockwise direction from $v_i$ to $v_j$ and other in anticlockwise direction from $v_i$ to $v_j$. So the number of such paths are $2 {n \choose 2}$. Thus we have $$F(C_n)=n+1+2{n \choose 2}=n^2+1$$

For $1\leq i<j\leq n,$  $f_{C_n}(v_i)=f_{C_n}(v_j).$ The single vertex $v_1$ is a connected subgraph of $C_n$ containing $v_1$. The cycle $C_n$ is also a connected subgraph of $C_n$ containing $v_1.$ All other connected subgraphs of $C_n$ containing $v_1$ are paths with at least two vertices. The number of paths in $C_n$ containing $v_1$ as a pendant vertex is  $2(n-1)$ and the number of paths in $C_n$ containing $v_1$ as non-pendant vertex is $n-1\choose 2$. Thus $f_{C_n}(v_1)=2n+ {n-1\choose 2}$.

\end{proof}

\begin{figure}[!h]
 \includegraphics[scale=0.75]{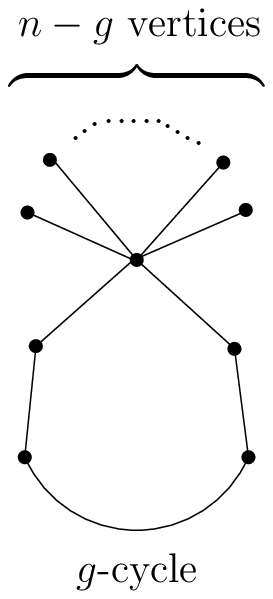}
\caption{The pineapple graph $U_{n,g}^p$}\label{fig6}
\end{figure}

For $3\leq g<n,$ $U_{n,g}(T_1,T_2,\ldots,T_g)$ denotes the unicyclic graph on $n$ vertices containing the cycle $C_g=:12 \cdots g1$ and trees $T_1,T_2,\ldots,T_g,$ where $T_i$ is a tree on $n_i+1$ vertices containing the only vertex $i$ of  $C_g$ for $i=1,2,\ldots,g.$ with $g+\sum {n_i}=n$. Then clearly, $U_{n,g}(T_1,T_2,\ldots,T_g)\in \mathcal{U}_{n,g}.$ If $T_i$ is a star on $n-g+1$ vertices for some $i= 1,2,\ldots,g$, then we call $U_{n,g}(T_1,T_2,\ldots,T_g)$ a pineapple graph and denote it by $U_{n,g}^p$(see Figure \ref{fig6}). 

\begin{figure}[!h]
 \includegraphics[scale=0.75]{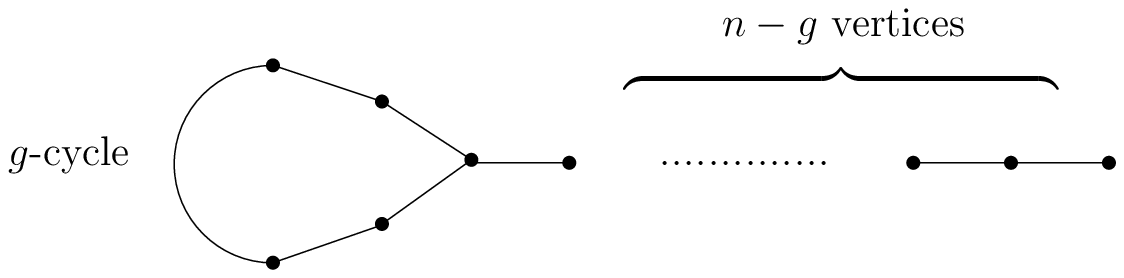}
\caption{The lollipop $U_{n,g}^l$}\label{fig7}
\end{figure}

If $T_i$ is a path on $n-g+1$ vertices with $i$ as one of its pendant vertex for some $i=1,2,\ldots,g$, then we call $U_{n,g}(T_1,T_2,\ldots,T_g)$ a lollipop graph graph and denote it as $U_{n,g}^l$(see Figure \ref{fig7}).

 \begin{lemma}\label{ucyclic1}
Let  $U_{n,g}(T_1,T_2,\ldots,T_g)$  be a unicyclic graph defined as above. Then $$F(U_{n,g}(P_{n_1+1},\cdots,P_{n_g+1}))\leq F(U_{n,g}(T_1,\ldots,T_g)) \leq F(U_{n,g}(K_{1,n_1},\ldots,K_{1,n_g}))$$
 where $K_{1,n_i}$ is the star on $n_i+1$ vertices with center at $i$ and $P_{n_i+1}$ is the path with $i$ as one of its pendant vertex.
\end{lemma}
 
 \begin{proof}
 Consider the unicyclic graph $U_{n,g}(T_1,T_2,\ldots,T_g)$. Suppose $T_i$ is not star with center at $i$. Then the vertex $i$ must be adjacent to a vertex, say $v$ of $T_i$ of degree at least $2.$  Identify the vertices $i$ and $v$ and add a pendant vertex at $i.$ Continue this operation till $T_i$ becomes $K_{1,n_i}$ with center at $i$, for $i=\{1,2,\ldots,g\}.$ By Lemma \ref{effect-1}, in each step the core index will increase. At last we get the unicyclic graph $U_{n,g}(K_{1,n_1},\ldots,K_{1,n_g}).$ So, $F(U_{n,g}(T_1,\ldots,T_g)) \leq F(U_{n,g}(K_{1,n_1},\ldots,K_{1,n_g}))$ and the equality  happens if and only if $U_{n,g}(T_1,\ldots,T_g) = U_{n,g}(K_{1,n_1},\ldots,K_{1,n_g}).$
 
To prove other inequality, suppose $T_k$ is not a path with $k$ as one of its pendant vertex. Then by using the grafting of edge operation, we can make $T_k$ a path with $k$ as one of its pendant vertex.
By Lemma \ref{effect-3}, in each step of this, the core index will decrease. Atlast we get the unicyclic graph $U_{n,g}(P_{n_1+1},\ldots,P_{n_g+1})).$ So, $F(U_{n,g}(P_{n_1+1},\ldots,P_{n_g+1}))\leq F(U_{n,g}(T_1,\ldots,T_g))$ and the equality  happens if and only if $U_{n,g}(T_1,\ldots,T_g) =U_{n,g} (P_{n_1+1},\ldots,P_{n_g+1}).$ This complete the proof.
 
 \end{proof}

\begin{theorem}\label{main-thm1}
For $3\leq g<n,$ let $G\in \mathcal{U}_{n,g}.$  Then $F(G)\leq F(U_{n,g}^p)$ with equality happens if and only if $G= U_{n,g}^p.$ Moreover, $$F(U_{n,g}^p)=n+g^2-g+1+(2^{n-g}-1)\left(2g+{g-1\choose 2}\right).$$
\end{theorem}
 \begin{proof}
 Any unicyclic graph on $n$ vertices with girth $g$ is isomorphic to $U_{n,g}(T_1,T_2,\ldots,T_g)$ for some trees $T_1,\ldots,T_g.$ Then by Lemma \ref{ucyclic1}, $F(G)\leq F(U_{n,g}(K_{1,n_1},\ldots,K_{1,n_g}))$ and equality happens if and only if $G= U_{n,g}(K_{1,n_1},\ldots,K_{1,n_g}).$  If exactly one vertex on the cycle of $U_{n,g}(K_{1,n_1},\ldots,K_{1,n_g})$ has degree greater than  $2$ then $U_{n,g}(K_{1,n_1},\ldots,K_{1,n_g})= U_{n,g}^p.$ Otherwise let $i$ and $j$ be  two vertices on the cycle of $U_{n,g}(K_{1,n_1},\ldots,K_{1,n_g})$ of degree greater than $2$. Without loss of generality, let $f_{U_{n,g}(K_{1,n_1},\ldots,K_{1,n_g})}(i)\geq f_{U_{n,g}(K_{1,n_1},\ldots,K_{1,n_g})}(j).$ Move the pendant vertices from the vertex $j$ to the vertex $i$. Continue this till exactly one vertex on the cycle of $U_{n,g}(K_{1,n_1},\ldots,K_{1,n_g})$ has degree greater than $2$. Then by Lemma \ref{effect-4}, $F(U_{n,g}(K_{1,n_1},\ldots,K_{1,n_g})\leq F(U_{n,g}^p)$ and equality happens if and only if $U_{n,g}(K_{1,n_1},\ldots,K_{1,n_g})= U_{n,g}^p.$ This proves the first part of the result.
 
 Furthermore, let $u$ be the vertex in $U_{n,g}^p$ with degree at least $3.$ Then
 \begin{align*}
 F(U_{n,g}^p)&=F(C_g)+F(K_{1,n-g})-1+(f_{C_g}(u)-1)(f_{K_{1,n-g}}(u)-1)\\
 &=g^2+1+2^{n-g}+n-g-1+\left(2g+{g-1 \choose 2} -1\right) (2^{n-g}-1 ) \\
 &=n+g^2-g+1+(2^{n-g}-1)\left(2g+{g-1\choose 2}\right).
 \end{align*}
\end{proof}

In Theorem \ref{main-thm1}, we proved that among all unicyclic graphs on $n$ vertices with  girth $g,$ the pineapple graph $U_{n,g}^p$  maximizes the core index. The corresponding reverse relation with the Wiener index is the following:
\begin{theorem}(\cite{yf}, Theorem 1.1)
For $3\leq g<n,$ among all unicyclic graphs on $n$ vertices with  girth $g$, the Wiener index is minimized by the pineapple graph $U_{n,g}^p.$
\end{theorem}

\begin{theorem}\label{main-thm2}
Let $G$ be a pineapple graph on $n\geq 4$ vertices. Then $$n^2+1+{n-2 \choose 2}\leq F(G)\leq (7 \times 2^{n-3})+n$$ with right equality happens if and only if $G= U_{n,3}^p$ and left equality happens if and only if $G= U_{n,n-1}^p.$
\end{theorem}
\begin{proof}

We first compare $F(U_{n,g}^p)$ and $F(U_{n,g+1}^p)$ for $g\geq 3.$ By Theorem \ref{main-thm1},
$$F(U_{n,g}^p)=n+g^2-g+1+(2^{n-g}-1)\left(2g+{g-1\choose 2}\right)$$ and
\begin{align*}
F(U_{n,g+1}^p)&=n+(g+1)^2-(g+1)+1+(2^{n-g-1}-1)\left(2g+2+{g \choose 2}\right)\\
         &= n+g^2+g+1+(2^{n-g-1}-1)\left(2g+2+{g \choose 2}\right).
\end{align*}
So, the difference

\begin{align*}
F(U_{n,g}^p)-F(U_{n,g+1}^p)&=-2g+(2^{n-g}-1)\left(2g+{g-1\choose 2}\right)-(2^{n-g-1}-1)\left(2g+2+{g \choose 2}\right)\\
                           &=-(g-1)+2^{n-g-1}\left[ \frac{g(g-1)}{2}\right]\\
                           &=(g-1)\left[2^{n-g-1}\left(\frac{g}{2}\right)-1\right]\\
		                       &>0, \mbox{ since $g\geq 3$ and $n> g$}.
\end{align*}
This implies $U_{n,3}^p$ has the maximum core index and $U_{n,n-1}^p$ has the minimum core index among all pineapple graphs on $n$ vertices and the result follows from Theorem \ref{main-thm1}. 
\end{proof}

We now determine the graph which maximizes the core index  among all unicyclic graph on $n$ vertices.

\begin{theorem}\label{main-thm3}
Among all unicyclic graph on $n\geq 4$ vertices the pineapple graph $U_{n,3}^p$ has the maximum core index.
\end{theorem}

\begin{proof}
Let $G$ be a unicyclic graph on $n\geq 4$ vertices. Suppose $G$ is not isomorphic to $C_n.$ Then by Theorem \ref{main-thm1} and Theorem \ref{main-thm2}, $F(G)\leq F(U_{n,3}^p)$ with equality happens if and only if $G\cong U_{n,3}^p.$ Since $F(C_n)=n^2+1$, $F(U_{n,3}^p)=(7 \times 2^{n-3})+n$ and $n\geq 4,$ so, $F(C_n)<F(U_{n,3}^p).$ Hence the result follows
\end{proof}
 The reverse relation of Theorem \ref{main-thm3} with Wiener index is the following:
\begin{theorem}(\cite{yf},Corollary 1.2)
Among all unicyclic graph on $n\geq 6$ vertices the pineapple graph $U_{n,3}^p$ has the minimum Wiener index.
\end{theorem}
For $n=4$ and $5$, over $\mathcal{U}_n$, the minimum Wiener index is attained by both the graphs $U_{n,3}^p$ and $C_n.$ Furthermore, $W(C_4)=8=W(U_{4,3}^p)$ and $W(C_5)=15=W(U_{5,3}^p).$ Next we characterize the graph which minimizes the core index over unicyclic graph on $n$ vertices. We will first do the minimization over unicyclic graphs with fixed girth.

 \begin{theorem}\label{main-thm4}
For $3\leq g<n,$ let $G\in \mathcal{U}_{n,g}.$ Then $F(G)\geq F(U_{n,g}^l)$ with equality happens if and only if $G= U_{n,g}^l.$ Moreover, $$F(U_{n,g}^l)=\left(\frac{n-g}{2}\right)(n+g^2+3)+g^2+1.$$
 \end{theorem}
 \begin{proof}
 Any unicyclic graph on $n$ vertices with girth $g$ is isomorphic to $U_{n,g}(T_1,T_2,\ldots,T_g)$ for some trees $T_1,\ldots,T_g.$ Then by Lemma \ref{ucyclic1}, $F(G)\geq F(U_{n,g}(P_{n_1+1},\ldots,P_{n_g+1}))$ and equality happens if and only if $G= U_{n,g}(P_{n_1+1},\ldots,P_{n_g+1}).$ If exactly one vertex on the cycle of $U_{n,g}(P_{n_1+1},\ldots,P_{n_g+1})$ has degree  $3$ then $U_{n,g}(P_{n_1+1},\ldots,P_{n_g+1})= U_{n,g}^l.$ Otherwise let $i$ and $j$ be  two vertices on the cycle of $U_{n,g}(P_{n_1+1},\ldots,P_{n_g+1})$ of degree  $3$. Without loss of generality, let $f_{U_{n,g}(P_{n_1+1},\ldots,P_{n_g+1})}(i)\leq f_{U_{n,g}(P_{n_1+1},\ldots,P_{n_g+1})}(j).$ Replace both the paths at $i$ and $j$ by a single path at $i$ on $n_i+n_j+1$ vertices.  Continue this till exactly one vertex on the cycle of $U_{n,g}(P_{n_1+1},\ldots,P_{n_g+1})$ has degree  $3$. Then by Lemma \ref{effect-5}, $F(U_{n,g}(P_{n_1+1},\ldots,P_{n_g+1}))\geq F(U_{n,g}^l)$ and equality happens if and only if $U_{n,g}(P_{n_1+1},...,P_{n_g+1})\cong U_{n,g}^l.$ This proves the first part of the result.
 
Furthermore, let $u$ be the only degree $3$ vertex of $U_{n,g}^l$ and let $v$ be the vertex in the path that adjacent to $u.$ Then
 \begin{align*}
 F(U_{n,g}^l)&=F(C_g)+F(P_{n-g})+f_{C_g}(u)f_{P_{n-g}}(v)\\
 &=g^2+1+{n-g+1 \choose 2}+\left(2g+ {g-1 \choose 2}\right)(n-g)\\
 &=\left(\frac{n-g}{2}\right)(n+g^2+3)+g^2+1
\end{align*}  
 \end{proof}
In Theorem \ref{main-thm4}, we proved that among all unicyclic graphs on $n$ vertices with  girth $g,$ the lollipop graph $U_{n,g}^l$  minimizes the core index. The corresponding reverse relation with the Wiener index is the following:
\begin{theorem}(\cite{yf},Theorem 1.1)
For $3\leq g<n,$ among all unicyclic graphs on $n$ vertices with  girth $g$, the Wiener index is maximized by the lollipop graph $U_{n,g}^l.$
\end{theorem}

The following result  compares the core index of  lollipop graphs.
\begin{theorem}\label{main-thm5}
Let $3\leq g< n$ and let $g_0$ be the largest positive integer such that  $\frac{3g_0^2-g_0+2}{2g_0}<n.$  Let $G$ be a lollipop graph on $n$ vertices. Then $$F(U_{n,3}^l)\leq F(G)\leq F(U_{n,g_0+1}^l)$$ with left equality happens if and only if $G= U_{n,3}^l$ and right equality happens if and only if $G= U_{n,g_0+1}^l.$
\end{theorem}

\begin{proof}
By Theorem \ref{main-thm4}, we have 
$$F(U_{n,g}^l)=\left(\frac{n-g}{2}\right)(n+g^2+3)+g^2+1$$ and 
$$F(U_{n,g+1}^l)=\left(\frac{n-g-1}{2}\right)(n+g^2+2g+4)+g^2+2g+2.$$
So the difference
\begin{align*}
F(U_{n,g+1}^l)-F(U_{n,g}^l) &=(2g+1)\left(\frac{n-g}{2}\right)-\left(\frac{n+g^2+2g+4}{2}\right)+2g+1\\  
		                        &=\frac{2gn-(3g^2-g+2) }{2}.
\end{align*}

Suppose $2gn-(3g^2-g+2)=0.$ Then $3g^2-(2n+1)g+2=0$ which implies $g=\frac{(2n+1)\pm\sqrt{(2n+1)^2-24}}{6}.$
But $g$ is an integer so $(2n+1)^2-24$ must be a perfect square. This implies $2n+1=7$ or $5$, which is a cotradiction as $n\geq 4.$ So either $F(U_{n,g+1}^l)-F(U_{n,g}^l)>0$ or $F(U_{n,g+1}^l)-F(U_{n,g}^l)<0.$

Let $g_0$ be the largest integer such that $\frac{3g_0^2-g_0+2}{2g_0}<n.$ Then the core index is maximized by the graph $U_{n,g_0+1}^l$ and minimized by the graph $U_{n,3}^l$ or $U_{n,n-1}^l.$ But 
$$F(U_{n,3}^l)=\left(\frac{n-3}{2}\right)(n+9+3)+9+1$$ and $$F(U_{n,n-1}^l)=\frac{1}{2}(n+(n-1)^2+3)+(n-1)^2+1.$$ So the difference $F(U_{n,n-1}^l)-F(U_{n,3}^l)=(n-3)(n-4)>0$ for $n\geq 5.$ Hence the result follows.
\end{proof}
\begin{theorem}\label{main-thm6}
Among all unicyclic graph on $n\geq 4$ vertices, the core index is minimized by  the lollipop graph $U_{n,3}^l$ if $n\geq 7$ and by the cycle $C_n$ if $n\leq5.$ For $n=6$, the core index is minimized by both the graphs $C_n$ and $U_{n,3}^l$.
\end{theorem}
\begin{proof}
We know $F(C_n)=n^2+1$ and $F(U_{n,3}^l)=\frac{n^2+9n-16}{2}$. So
\begin{equation*}
(n^2+1) - \frac{n^2+9n-16}{2}
\begin{cases}
>0, &\text{if $n\geq 7$}\\
=0, &\text{if $n=6$}\\
<0, &\text{if $n\leq 5$.}
\end{cases}
\end{equation*}

Hence by Theorem \ref{main-thm4} and Theorem \ref{main-thm5},  the result follows.
\end{proof}
We end this section with the result corresponding to Theorem \ref{main-thm6} associated with Wiener index.
\begin{theorem}(\cite{yf}. Corollary 1.2)
Among all unicyclic graph on $n\geq 5$ vertices the lollipop graph $U_{n,3}^l$ has the maximum Wiener index.
\end{theorem}
For $n=4$, the graphs $U_{4,3}^l$ and $C_4$ are the only elements of $\mathcal{U}_4$ and $W(C_4)=8=W(U_{4,3}^l).$
 
\section{Graphs with fixed number of pendant vertices}

Throughout this section, graphs are all connected. Our goal in this section is to characterize the graphs which extremizes the core index over all connected graphs of order $n$ with $k$ pendant vertices. Let $\mathfrak{H}_{n,k}$ denote the class of all connected graphs of order $n$ with $k$ pendant vertices. For $k=n,$  $\mathfrak{H}_{n,k}=\{K_2\}$. If $k=n-1$ then $n\geq 3$ and $\mathfrak{H}_{n,k}=\{K_{1,n-1}\}$.  For $n=3,$  either $k=0$($C_3$ is the only graph in this case) or $k=2$($K_{1,2}$ is the only graph in this case). So we assume that $0\leq k\leq n-2$ and $n\geq 4$. We first consider the problem of maximizing the core index over $\mathfrak{H}_{n,k}$. 
  
For $0\leq k\leq n-3$, let $P_n ^k$ be the graph  obtained by adding $k$ pendant vertices  to a single vertex of the complete graph $K_{n-k}$. Then $P_n ^k \in \mathfrak{H}_{n,k}.$
\begin{theorem}\label{pmax-thm1}
For $0\leq k\leq n-3,$ The graph $P_n^k$ maximizes the core index
over $\mathfrak{H}_{n,k}.$ Furthermore, $F(P_n^k)=(2^k-1)(F(K_{n-k})-F(K_{n-k-1}))+F(K_{n-k})+k$
\end{theorem}
\begin{proof}
Let  $G\in \mathfrak{H}_{n,k}$ and let $v_1,v_2,\ldots,v_{n-k}$ be the non-pendant vertices of $G.$ If the induced subgraph $G[v_1,v_2,\ldots,v_{n-k}]$ is not complete, then form a new graph $G'$ from $G$ by joining all the non-adjacent non-pedant vertices of $G$ with new edges. Then $G' \in \mathfrak{H}_{n,k}$ and $F(G)<F(G').$ If $G'= P_n ^k$ then we are done, otherwise $G'$ has at least two vertices of degree greater than or equal to $n-k.$  Form a new graph $G''$ from $G'$ by moving all the pendant vertices to one of the vertex $v_1,v_2,\ldots,v_{n-k}$ following the pattern mentioned in the statement of the Lemma \ref{effect-4}. Then  $G'' = P_n ^k$ and by Lemma \ref{effect-4}, the result follows. 

For any vertex $v$ of the complete graph $K_n$, $f_{K_n}(v)=F(K_n)-F(K_{n-1}).$ Also we know, $F(K_{1,k})=2^k+k$ and $f_{K_{1,k}}(u)=2^k$ where $u$ is the non-pendant vertex of $K_{1,k}.$ Let $w$ be the vertex of $P_n^k$ with which $k$ pendant vertices are adjacent. Then 
\begin{align*}
F(P_n^k)&=F(K_{n-k})+F(K_{1,k})-1+(f_{K_{n-k}}(w)-1)(f_{K_{1,k}}(w)-1)\\
&= F(K_{n-k})+F(K_{1,k})+f_{K_{n-k}}(w)f_{K_{1,k}}(w)-f_{K_{n-k}}(w)-f_{K_{1,k}}(w)\\
&=(2^k-1)(F(K_{n-k})-F(K_{n-k-1}))+F(K_{n-k})+k
\end{align*}
This completes the proof.
\end{proof}

\begin{figure}[h]
\centering
\includegraphics[scale=.9]{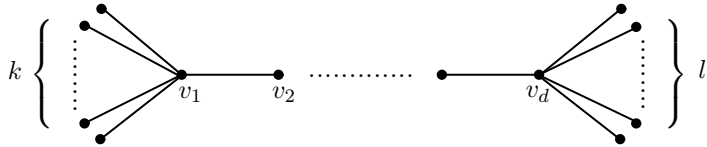}
\caption{The tree $T(k,l,d)$}\label{The tree T(k,l,d)}
\end{figure}

For a fixed positive integer $n,$ the path $ [v_1 v_2\cdots v_n]$ on $n$
vertices is denoted by $P_n$. For positive integers $k,l,d$ with
$n=k+l+d$, let $T(k,l,d)$ be the tree of order $n$ obtained by taking the
 path  $P_d$  and adding $k$ pendant vertices
adjacent to $v_1$ and $l$ pendant vertices adjacent to $v_d$ (see Figure
\ref{The tree T(k,l,d)}). Note that $T(1,1,d)$ is a path on $d+2$
vertices. 
\begin{theorem}\label{pmax-thm2}
The tree $T(1,n-3,2)$ maximizes the core index
over $\mathfrak{H}_{n,n-2}.$ Furthermore, $F(T(1,n-3,2))=3(2^{n-3})+n.$
\end{theorem}
\begin{proof}
Let $G\in \mathfrak{H}_{n,n-2}.$ Then $G$ is isomorphic to $T(k,l,2)$ for some $k,l\geq 1.$ If $k,l\geq 2$ then form the tree $T(1,n-3,2)$ from $G$ by moving pendant vertices from one vertex to other following the pattern mentioned in the statement of the Lemma \ref{effect-4}. Then $F(G)<F(T(1,n-3,2))$ and also
$F(T(1,n-3,2))=3+2^{n-3}+n-3+2^{n-2}=3(2^{n-3})+n$. This completes the proof.
\end{proof}

Let $\mathfrak{T}_{n,k}$ denote the subclass of $\mathfrak{H}_{n,k}$
consisting of all trees of order $n$ with $k$, $2\leq k\leq n-2$, pendant
vertices. By Theorem \ref{pmax-thm2}, the tree $T(1,n-3,2)$ has the maximum core index over $\mathfrak{T}_{n,n-2}$. We next prove this result for $2\leq k\leq n-3.$ To do so, we first define the following tree.

Let $T_{n,k}$ be the tree of order $n$ that has a vertex $v$ of degree $k$
and $T_{n,k}- v = r P_{q+1} \cup (k-r) P_q$, where $q=\lfloor
\frac{n-1}{k} \rfloor$ and $r = n-1 -kq$.

\begin{figure}[h]
\centering
\includegraphics[scale=.9]{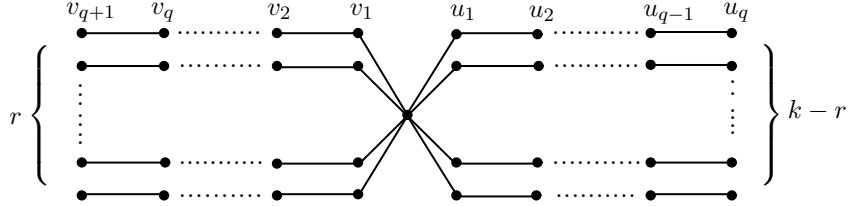}
\caption{The tree $T_{n,k}$}\label{starpath tree}
\end{figure}

\begin{theorem}\label{pmax-thm3}
For $2\leq k\leq n-3,$  the tree $T_{n,k}$ maximizes the core index over $\mathfrak{T}_{n,k}$. Furthermore, $$F(T_{n,k})=(q+2)^r(q+1)^{k-r}+\frac{(q+1)(qk+2r)}{2}.$$
\end{theorem} 
\begin{proof}
If $k=2$ then $P_n$ is the only tree in $\mathfrak{T}_{n,2}$ and the result is true. So assume $k\geq 3.$ Let $T\in \mathfrak{T}_{n,k}.$ Suppose $T$ has at least two vertices of degree greater than or equal to $3.$ Let $u$ and $v$ be two vertices of $T$ with $d(u),d(v) \geq 3$ and every vertex in the $u$-$v$ path of $T$ except $u$ and $v$(if any) is of degree $2.$ Form a new tree $T'$ by identifying  two vertices of the path and adding a pendant vertex to a pendant vertex of $T$ following the pattern mentioned in the Lemma \ref{effect-2}. The new tree $T'\in \mathfrak{T}_{n,k}$ and by Lemma \ref{effect-2}, $F(T)<F(T').$ Continue this till the identification of $u$ and $v$. In this process, we will get a tree $\tilde{T}\in\mathfrak{T}_{n,k}$ with exactly one vertex, say $w$ of degree $k$ and $F(T)<F(\tilde{T}).$ If $\tilde{T} = T_{n,k}$ then we are done. Otherwise  $\tilde{T}-w$ has $k$ paths and at least two of its paths $P_{n_1}$ and $P_{n_2}$ with $n_1-n_2\geq 2.$ By using the grafting of edge operation on $\tilde{T},$ we can reduce the difference of the length of the paths of $\tilde{T}-w$ to at most $1$. In this process at last we get the tree  $T_{n,k}$ and by Lemma \ref{effect-3}, every step the core index will increase. This completes the first part of the proof.\\

 For $l\geq 2$ and $q\geq 1$, let $T_l^q$ be the tree on $lq+1$ vertices with $l$ pendant vertices having one vertex $v$ of degree $l$ and $T_l^q-v=lP_q$. We will first find the value of  $F(T_l^q)$ and use that to get the value of $F(T_{n,k})$. In this context, we denote the path on $q+1$ vertices by $T_1^q$ and for a pendant vertex $v$ of $T_1^q,$ $f_{T_1^q}(v)=q+1.$  For $l \geq 2,$
$f_{T_l^q}(v)=f_{T_{l-1}^q}(v)+qf_{T_{l-1}^q}(v)=(q+1)f_{T_{l-1}^q}(v).$
So, we get $f_{T_{l}^q}(v)=(q+1)^l$ for $l\geq 1.$ Then
\begin{align*}
F(T_l^q)&=F(T_{l-1}^q)+F(P_q)+qf_{T_{l-1}^q}(v)\\
&=F(T_{l-1}^q)+\frac{q(q+1)}{2}+q(q+1)^{l-1}. 
\end{align*}
Solving this recurrence relation, we get $F(T_l^q)=(q+1)(\frac{lq}{2}+(q+1)^{l-1})$.

For the tree $T_{n,k},$ we have $n-1=kq+r, \; 0\leq r< k.$ When $r=0,$ $T_{n,k}$ is isomorphic to $T_k^q$ and $F(T_{n,k})=(q+1)(\frac{kq}{2}+(q+1)^{k-1})=(q+1)^k+\frac{kq(q+1)}{2}.$

For $1\leq r < k,$ let $v$ be the vertex of degree $k$ in $T_{n,k}.$ Then we have 
\begin{align*}
F(T_{n,k})&=F(T_r^{q+1})+F(T_{k-r}^q)-1+(f_{T_r^{q+1}}(v)-1)(f_{T_{k-r}^q}(v)-1)\\
&=(q+2)\left(\frac{r(q+1)}{2}+(q+2)^{r-1}\right)+(q+1)\left(\frac{(k-r)q}{2}+(q+1)^{k-r-1}\right)-1\\
&+((q+2)^r-1)((q+1)^{k-r}-1)\\
&=(q+2)^r(q+1)^{k-r}+\frac{(q+1)(qk+2r)}{2}.
\end{align*}
This completes the proof.
\end{proof}

The above result is proved by Zhang et al. in \cite{zha} (see Corollary 5.3) but our proof is little different. We also explained the counting for $F(T_{n,k})$. We will now consider the problem of minimizing the core index over $\mathfrak{H}_{n,k}.$ The following is an important result in this regard.

\begin{theorem}\label{pmax-thm4}(\cite{sli}, Theorem 1)
For $2\leq k\leq n-2,$  the tree $T(\lfloor\frac{k}{2}\rfloor,\lceil\frac{k}{2}\rceil,n-k)$ minimizes the core index over $\mathfrak{T}_{n,k}$. Further 
\begin{equation*}
F(T(\lfloor\frac{k}{2}\rfloor,\lceil\frac{k}{2}\rceil,n-k))=
\begin{cases}
(n-k-1)2^{\frac{k}{2}+1}+2^k+k+{n-k-1 \choose 2}, & \textit{if k is even}\\ 
3(n-k-1)2^{\frac{k-1}{2}}+2^k+k+{n-k-1 \choose 2}, & \textit{if k is odd.}
\end{cases}
\end{equation*}
\end{theorem}

\begin{theorem}\label{pmax-thm5}
For $2\leq k\leq n-2,$ if $G\in \mathfrak{H}_{n,k}$ then $F(G)\geq F(T(\lfloor\frac{k}{2}\rfloor,\lceil \frac{k}{2}\rceil,n-k))$ and equality happens if and only if $G= T(\lfloor\frac{k}{2}\rfloor,\lceil \frac{k}{2}\rceil,n-k).$
\end{theorem}
\begin{proof}
Let $G \in \mathfrak{H}_{n,k}.$  Construct a spanning tree $G'$ from $G$ by deleting some edges. Clearly $F(G')<F(G)$. The number of pendent vertices of $G'$ is greater than or equal to $k$.  Suppose $G'$ has more than $k$ pendant vertices. Since $k\geq 2$, $G'$ has at least one vertex of degree greater than $2.$

Consider a vertex $v$ of $G'$ with $d(v)\geq 3$ and two paths $P_{l_1},P_{l_2},\; l_1\geq l_2$ attached at $v.$ Using grafting of edge operation on $G'$, we will get a new tree $\tilde{G}$ with number of pendant vertices one less than the number of pendant vertices of $G'$ and by Lemma \ref{effect-3}, $F(\tilde{G})<F(G').$ Continue this process till we get a tree with $k$ pendant vertices from $\tilde{G}.$ By Lemma \ref{effect-3}, every step in this process the core index will decrease. So, we will reach at a tree of order $n$ with $k$ pendant vertices and the result follows from Theorem \ref{pmax-thm4}.
\end{proof}

\begin{theorem}\label{pmax-thm6}
The lollipop graph $U_{n,3}^l$ minimizes the core index over $\mathfrak{H}_{n,1}.$ 
\end{theorem}

\begin{proof}
Let $G\in \mathfrak{H}_{n,1}.$ Since $G$ is connected and has exactly one pendent vertex, it must contain a cycle. Let $C_g$ be a cycle in $G.$ If $G$ has more than one cycle, then construct a new graph $G'$ from $G$ by deleting edges from all cycles other than $C_g$ so that the graph remains connected. Clearly $F(G')<F(G)$ and $G'$ is a unicyclic graph on $n$ vertices with girth $g.$ By Theorem \ref{main-thm4} and Theorem \ref{main-thm5}, $F(U_{n,3}^l)\leq F(G')$ and equality happens if and only if $G' = U_{n,3}^l.$ As $U_{n,3}^l\in \mathfrak{H}_{n,1}$, so the result follows. 
\end{proof}

The only case left in the problem of minimizing the core index over $\mathfrak{H}_{n,k}$ is when $k=0.$ For $k=0,$  $n$ must be at least $3.$ The cycle $C_3$ is the only element of $\mathfrak{H}_{3,0}.$ With respect to isomorphism, there are only three connected  graphs on $4$ vertices without any pendant vertex. It can be easily checked that $C_4$ has  the minimum core index over $\mathfrak{H}_{4,0}.$ For the rest of this section, $n$ is at least $5.$

\begin{figure}[h]
\centering
\includegraphics[scale=.9]{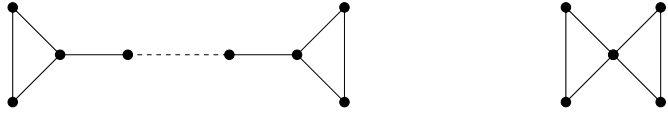}
\caption{The graphs $C_{3,3}^n$ and $C_{3,3}^5$}\label{fig:10}
\end{figure}

Let $m_1,m_2\geq 3$ be two integers. Let $n$ be an integer such that $n\geq m_1+m_2-1.$ Take a path on $n-(m_1+m_2)+2$ vertices and identify one pendant vertex of the path with a vertex of $C_{m_1}$ and another pendant vertex with a vertex of $C_{m_2}.$ If $n= m_1+m_2-1$  then identify one vertex of $C_{m_1}$ with one vertex of $C_{m_2}.$ We denote this graph by $C_{m_1,m_2}^n$ (see Figure \ref{fig:10}). Clearly $C_{m_1,m_2}^n$ is a connected graph on $n$ vertices without any pendant vertex. The next lemma compares the core index of $C_{3,3}^n$ and the cycle $C_n.$

\begin{lemma}\label{3cycles}
For $n\geq 6,$ $F(C_n)<F(C_{3,3}^n)$ if and only if $n \leq 16$.
\end{lemma}

\begin{proof}
For $n=6,$ let $v$ and $w$ be two vertices of $C_{3,3}^6$ with $d(v)=d(w)=3.$ Then $F(C_{3,3}^6)=F(C_3)+F(C_3)+f_{C_3}(v)f_{C_3}(w)>37=F(C_n).$ For $n\geq 7,$ let $v$  be a vertex of one of the  $3$-cycles with degree $3$ and let $w$ be the vertex adjacent to $v$ but not in that $3$-cycle of $C_{3,3}^n,$. Then

\begin{align*}
F(C_{3,3}^n)&=F(C_3)+F(U_{n-3,3}^l)+f_{C_3}(v)f_{U_{n-3,3}^l}(w)\\
&=10+\left(\frac{n-6}{2}\right)(n+9)+10+7(n+1)\\
&=\frac{n^2+17n}{2}
\end{align*}
So,
$F(C_{3,3}^n)-F(C_n)=\frac{n^2+17n}{2}-n^2-1=\frac{17n-n^2-2}{2}>0 $ if and only if  $n\leq 16.$ This completes the proof.
\end{proof}

\begin{lemma}\label{twocycles}
Let $m_1,m_2\geq 3$ be two integer and let $n=m_1+m_2-1.$  Then $F(C_n)<F(C_{m_1,m_2}^n)$.
\end{lemma}
\begin{proof}
 Let $v$ be the vertex of degree $4$ in $C_{m_1,m_2}^n.$ Then
\begin{align*}
F(C_{m_1,m_2}^n)&=F(C_{m_1})+F(C_{m_2})-1+(f_{C_{m_1}}(v)-1)(f_{C_{m_2}}(v)-1)\\
&=m_1^2+1+m_2^2+1-1+\left(2m_1+ {m_1-1 \choose 2} -1\right)\left(2m_2+ {m_2-1 \choose 2} -1\right)\\
&\geq m_1^2+m_2^2+1+4m_1m_2.
\end{align*}

So, the difference $F(C_{m_1,m_2}^n)-F(C_n) \geq 2m_1m_2+2m_1+2m_2-1>0.$
\end{proof}

\begin{corollary}\label{ctwo}
Let $m_1,m_2\geq 3$ be two integer and let $n=m_1+m_2-1.$ Let $G\in \mathfrak{H}_{n,0}$ with $C_{m_1,m_2}^n$ as a subgraph of $G.$ Then $F(G)>F(C_n).$
\end{corollary}

\begin{lemma}\label{24}
Let $u$ be the pendant vertex  and $v$ be a non-pendant vertex of the unicyclic graph $U_{n,g}^l$. Then $f_{U_{n,g}^l}(u)<f_{U_{n,g}^l}(v)$.
\end{lemma}
\begin{proof}
Let $g$ be the vertex of degree $3$ in $U_{n,g}^l$ and let $g+1$ be the  vertex adjacent to $g$ not on the $g$-cycle of $U_{n,g}^l$. Then $$f_{U_{n,g}^l}(u)=f_{P_{n-g}}(u)+f_{C_g}(g).$$

If $v$ is a vertex on the cycle $C_g$ of $U_{n,g}^l$ then

$$f_{U_{n,g}^l}(v)=f_{C_g}(v)+f_{C_g}(v,g)f_{P_{n-g}}(g+1)$$

and if $w$ is a non pendant vertex  of $U_{n,g}^l$ which is not on the cycle then

$$f_{U_{n,g}^l}(w)=f_{P_{n-g}}(w)+f_{C_g}(g)f_{P_{n-g}}(g+1,w).$$

Since $f_{P_{n-g}}(g+1)=f_{P_{n-g}}(u)$ and $f_{P_{n-g}}(w)\geq f_{P_{n-g}}(u),$ so
\begin{align*}
f_{U_{n,g}^l}(v)-f_{U_{n,g}^l}(u)&=f_{P_{n-g}}(u)(f_{C_g}(v,g)-1)>0.\\
f_{U_{n,g}^l}(w)-f_{U_{n,g}^l}(u)&=f_{P_{n-g}}(w)-f_{P_{n-g}}(u)+f_{C_g}(g)(f_{P_{n-g}}(g+1,w)-1)>0.
\end{align*}
\end{proof}

\begin{lemma}\label{3uni1}
Let $u$ be a vertex of a connected graph $G$ on at least two vertices. Suppose $v$ is the pendant vertex of $U_{n,g}^l$ and $w$ is a non-pendant vertex of $U_{n,g}^l$. Let $G_1$ be the graph obtained from $G$ and $U_{n,g}^l$ by identifying $u$ with $v$ and $G_2$ be the graph obtained by identifying $u$ with $w$. Then $F(G_1)<F(G_2).$ 
\end{lemma}
\begin{proof}
 We have
\begin{align*}
F(G_1)&=F(G)+F(U_{n,g}^l)-1+(f_{G}(u)-1))(f_{U_{n,g}^l}(v)-1)\\
F(G_2)&=F(G)+F(U_{n,g}^l)-1+(f_{G}(u)-1)(f_{U_{n,g}^l}(w)-1).
\end{align*}
So, the difference
\begin{align*}
F(G_2)-F(G_1)=(f_{G}(u)-1)(f_{U_{n,g}^l}(w)-f_{U_{n,g}^l}(v))>0 .
\end{align*}
The last inequality follows from Lemma \ref{24}.
\end{proof}

\begin{theorem}\label{mainp-1}
Let $G\in \mathfrak{H}_{n,0}$ with at least one cut-vertex. Suppose $C_{m_1,m_2}^n$ with $m_1+m_2-1=n$ is  not a subgraph of $G$. Then $F(G)\geq F(C_{g_1,g_2}^n)$ for some $g_1,g_2\geq 3$ and the equality holds if and only if $G=C_{g_1,g_2}^n.$
\end{theorem}
\begin{proof}
Since $G$ has a cut-vertex and no pendant vertices, so $G$ contains two cycles with at most one common vertex. Let $C_{g_1}$ and $C_{g_2}$ be two cycles of $G$ with at most one common vertex. Since $C_{m_1,m_2}^n$ with $m_1+m_2-1=n$ is  not a subgraph of $G$, so $g_1+g_2\leq n.$ Clearly $G$ has at least $n+1$ edges.

If $G$ has exactly $n+1$ edges, then there is no common vertex between $C_{g_1}$ and $C_{g_2}$ and $G=C_{g_1,g_2}^n.$ So, let $G$ has at least $n+2$ edges. Suppose $|E(G)|=n+k$, where $k\geq 2.$ Choose $k-1$ edges $\{e_1,\ldots,e_{k-1}\}\subset E(G)$ such that $e_i\notin E(C_{g_1})\cup E(C_{g_2}),\;i=1,\ldots, k-1$ and $G\setminus \{e_1,\ldots,e_{k-1}\}$ is connected. Let $G_1=G\setminus \{e_1,\ldots,e_{k-1}\}$ ($G_1$ may have some pendant vertices). Then $F(G_1)< F(G).$ If $G_1$ has no pendant vertices then $G_1=C_{g_1,g_2}^n.$

Let $G_1$ has some pendant vertices. Then for some $l<n,$ $C_{g_1,g_2}^l$ is a subgraph of $G_1.$ By grafting of edges operation(if required), we can form a new graph $G_2$ from $G_1$ where $G_2$ is a connected graph on $n$ vertices obtained by attaching some paths to some vertices of $C_{g_1,g_2}^l.$ Then by Lemma \ref{effect-3}, $F(G_2)< F(G_1).$ If more than one paths are attached to different vertices of $C_{g_1,g_2}^l$ in $G_2$, then using the graph operation as mentioned in Lemma \ref{effect-5}, form a new graph $G_3$ from $G_2$, where $G_3$ has exactly one path attached to $C_{g_1,g_2}^l.$ Then by Lemma \ref{effect-5}, $F(G_3)< F(G_2).$ 

Let the path attached to the vertex $u$ in $C_{g_1,g_2}^l$ of $G_3.$ Then we have two cases:\\
\textbf{Case-1:} $u \in V(C_{g_1})\cup V(C_{g_2})$\\
Without loss of generality, assume that $u\in V(C_{g_1}).$ Then the induced subgraph of $G_3$ containing the vertices of $C_{g_1}$ and the vertices of the path attached to it,  is the lollipop graph $U_{k,g_1}^l$ for some $k>g_1.$ Let $v$ be the pendant vertex of $U_{k,g_1}^l$. Since the two cycles $C_{g_1}$ and $C_{g_2}$ have at most one vertex in common, so we have two subcases:

\underline{Subcase-1}: $V(C_{g_1})\cap V(C_{g_2})=\{w\}$ 

Let $H_1$ be the induced subgraph of $G_3$ containing the vertices $\{V(G_3)\setminus V(U_{k,g_1}^l)\}\cup \{w\}.$ Clearly $H_1$ is the cycle $C_{g_2}.$ Then identify the vertex $v$ of $U_{k,g_1}^l$ with the vertex $w$ of $H_1$ to form a new graph $G_4.$  By Lemma \ref{3uni1}, $F(G_4)< F(G_3)$ and $G_4$ is the graph $C_{g_1,g_2}^n.$

\underline{Subcase-2}: $V(C_{g_1})\cap V(C_{g_2})=\phi$

Let $H_2$ be the induced subgraph of $G_3$ containing the vertices $V(G_3)\setminus V(U_{k,g_1}^l).$ In $G_3$ exactly one vertex $w_1\in U_{k,g_1}^l$ adjacent to exactly one vertex $w_2$ of $H_2.$ Form a new graph $G_5$ from $G_3$ by deleting the edge $\{w_1,w_2\}$ and adding the edge $\{v,w_2\}.$ By Lemma \ref{3uni1}, $F(G_5)< F(G_3)$ and $G_5$ is the graph $C_{g_1,g_2}^n.$\\

\textbf{Case-2:} $u \notin V(C_{g_1})\cup V(C_{g_2})$\\
Let $w$ be the pendant vertex of $G_3$ and let $w_3$  be a vertex in $C_{g_1,g_2}^l$ of $G_3$ adjacent to $u.$ Form a new graph $G_6$ from $G_3$ by deleting the edge $\{u,w_3\}$ and adding the edge $\{w,w_3\}.$ By Lemma \ref{3uni1}, $F(G_6)< F(G_3)$ and $G_6$ is the graph $C_{g_1,g_2}^n.$ This completes the proof.
\end{proof}

\begin{lemma}\label{3uni2}
Let $u$ be a vertex of a connected graph $G.$ For $m\geq 4,$ let $G_1$ be the graph obtained by identifying the vertex $u$ of $G$ with the pendant vertex of $U_{m+1,m}^l$  and $G_2$ be the graph obtained by identifying the vertex $u$ with the pendant vertex of $U_{m+1,3}^l$. Then $F(G_2)<F(G_1)$.
\end{lemma}
\begin{proof}
We have
\begin{align*}
F(G_1)&=F(G)+F(U_{m+1,m}^l)-1+(f_G(u)-1)(f_{U_{m+1,m}^l}(u)-1)\\
F(G_2)&=F(G)+F(U_{m+1,3}^l)-1+(f_G(u)-1)(f_{U_{m+1,3}^l}(u)-1).
\end{align*}
By Theorem \ref{main-thm5}, $F(U_{m+1,3}^l)<F(U_{m+1,m}^l)$. So, the difference 
\begin{align*}
F(G_1)-F(G_2)&>(f_G(u)-1)(f_{U_{m+1,m}^l}(u)-f_{U_{m+1,3}^l}(u))\\
&=(f_G(u)-1)\left(1+2m+{m-1 \choose 2}-m+2-7\right) \\ 
&=(f_G(u)-1)\left(m-4+{m-1 \choose 2}\right)\\
&>0
\end{align*}
\end{proof}

\begin{corollary}\label{c33}
Let $m_1,m_2\geq 3$ be two integers and let $m_1+m_2\leq n.$ Then $F(C_{m_1,m_2}^n)\geq F(C_{3,3}^n)$ and equality happens if and only if $m_1=m_2=3.$
\end{corollary}

\begin{theorem}\label{mainp-2}
Let $G$ be a $2$-connected graph on $n\geq 5$ vertices. Then $F(G) \geq F(C_n)$ and the equality holds if and only if $G=C_n$.
\end{theorem}
\begin{proof}
Let $g$ be the circumference of $G$. Let $C_g$ be a $g$-cycle in $G$. Then every connected subgraph of $C_g$ is also a connected subgraph of $G$. 
If $g=n$ and $G$ is not isomorphic to $C_n$, then $G$ has at least $n+1$ edges. In this case clearly $F(G)>F(C_n).$ 

If $g=n-1$ then the number of connected subgraphs of $C_g$ is equal to $(n-1)^2+1.$ Let $v$ be the vertex of $G$ not on the cycle $C_g$ of $G.$ Since $G$ is  $2$-connected, so $v$ is adjacent to at least two vertices of $C_g$. Let $u$ be a vertex of $C_g$ such that $\{u,v\}\in E(G).$ Then $C_g\cup \{u,v\}$ is a connected subgraph of $G.$ By Lemma \ref{cycle-F}, the number of connected subgraphs of $C_g\cup \{u,v\}$ containing $\{u,v\}$ is $2(n-1)+ {n-2\choose 2}.$ So, $F(G)>(n-1)^2+1+2(n-1)+ {n-2\choose 2}>n^2+1=F(C_n).$

If $g\leq n-2$ then at least two vertices of $G$ are not on the cycle $C_g$. Since $G$ is  $2$-connected, so every pair of distinct vertices  $u,v \in V(G)\setminus V(C_g)$ there exists at least two distinct paths in $G$ with $u$ and $v$ as pendant vertices. Each of these paths is a connected subgraph of $G$. Apart from these subgraphs, also for every $v \in V(G)\setminus V(C_g)$ there exists a $w \in V(C_g)$ such that there is a path joining $v$ and $w$. This path together with $C_g$ forms a lolipop subgraph of $G$ with $v$ as a pendant vertex. Thus there are at least $(n-g)(f_{C_g}(w))$ more connected subgraphs in $G$ different from the above mentioned connected subgraphs of $G.$ Thus

\begin{align*}
F(G) &\geq F(C_g)+2{n-g \choose 2}+(n-g)(f_{C_g}(w))\\
&=g^2+1+(n-g)(n-g-1)+(n-g)\left(2g+{g-1 \choose 2}\right)\\
&=n^2+1+{\frac{g(n-g)}{2}}(g-3)> n^2+1=F(C_n).
\end{align*}
The last inequality follows from the fact that $g$ is the circumference of a $2$-connected graph on $n\geq 5$ vertices. Hence the result follows. 
\end{proof}

The above results leads to the  main theorem of the section.
\begin{theorem}\label{mlast}
Let $G\in \mathfrak{H}_{n,0}$. Then 
\begin{equation*}
F(G)
\begin{cases}
\geq F(C_n), &\text{if $n\leq 16$}\\
\geq F(C_{3,3}^n), &\text{if $n> 16$.}
\end{cases}
\end{equation*}
Moreover, $F(G)\geq \min\{n^2+1, \frac{n^2+17n}{2}\}.$ 
\end{theorem}
\begin{proof}
If $G$ has no cut-vertices, the by Theorem \ref{mainp-2} and Lemma \ref{3cycles} the result follows. Suppose $G$ has a cut-vertex. If $C_{m_1,m_2}^n$ with $m_1+m_2-1=n$ is a subgraph of $G$ then by Corollary \ref{ctwo} and  and Lemma \ref{3cycles} the result follows. If $C_{m_1,m_2}^n$ with $m_1+m_2-1=n$ is not a subgraph of $G$ then by Theorem \ref{mainp-1},Corollary \ref{c33} and Lemma \ref{3cycles}, the result follows.   
\end{proof}

\vskip .5cm

\noindent{\bf Addresses}:\\

\noindent Dinesh Pandey, Kamal Lochan Patra\\

\noindent 1) School of Mathematical Sciences\\
National Institute of Science Education and Research (NISER), Bhubaneswar\\
P.O.- Jatni, District- Khurda, Odisha - 752050, India\medskip

\noindent 2) Homi Bhabha National Institute (HBNI)\\
Training School Complex, Anushakti Nagar\\
Mumbai - 400094, India\medskip

\noindent E-mails: dinesh.pandey@niser.ac.in, klpatra@niser.ac.in\\
\end{document}